\documentclass[12pt]{amsart}
\usepackage{amsfonts}
\usepackage{amssymb}
\usepackage{mathtools}
\usepackage[all]{xy}
\usepackage{amsmath, amssymb, amsfonts, amstext, amsthm}
\usepackage[mathscr]{euscript}
\usepackage{enumerate}
\usepackage{verbatim}
\usepackage{hyperref}
\usepackage{braket}
\usepackage{thmtools, thm-restate}
\theoremstyle{plain}
\def \ra{\rightarrow}
\def \mfk{\mathfrak}
\def \mcO{\mathcal{O}}
\def \mcOP{\mcO_{\mbb{P}^2}}

\def \mcD{\mathcal{D}}

\def \mcC{\mathcal{C}}

\def \mbb{\mathbb}

\def \xhra{\xhookrightarrow}

\newcommand{\trm}{\textrm}

\newcommand{\hra}{\hookrightarrow}
\newcommand{\xra}{\xrightarrow}
\usepackage{hyperref}

\title{Ulrich line bundles on double planes}
\author[A.~J.~Parameswaran]{A.~J.~Parameswaran}
\address{School of Mathematics, Tata Institute of Fundamental Research, Mumbai - 400005.}
\email{param@math.tifr.res.in}

\author[P.~Narayanan]{Poornapushkala Narayanan}
\address{School of Mathematics, Tata Institute of Fundamental Research, Mumbai - 400005 \\ Current address: International Centre for Theoretical Sciences, Bengaluru - 560089.}
\email{poorna.narayanan@icts.res.in}

\thanks{Mathematics Classification numbers: 14E20, 14J60, 14H50, 14C20}

\begin{document}

\begin{abstract}
  Consider a smooth complex surface $X$ which is a double cover of the
  projective plane $\mbb{P}^2$ branched along a smooth curve of degree
  $2s$. In this article, we study the geometric conditions which are
  equivalent to the existence of Ulrich line bundles on $X$ with
  respect to this double covering. Also, for every $s\geq 1$, we
  describe the classes of such surfaces which admit Ulrich line
  bundles and give examples.
\end{abstract}

\maketitle

\section{Introduction}\label{introduction}
Suppose that $X\xhra{i} \mbb{P}^N$ is a smooth projective variety over
the field of complex numbers. A vector bundle $E$ on $X$ is said to be
\emph{Ulrich} 
if for any finite linear projection
${\pi:X\ra \mbb{P}^{\text{dim}\,X}}$ obtained 
from $i$, the direct image $\pi_*E$ is the trivial vector bundle on
the projective space, cf.  $\mathcal{x}$ \ref{prelim-ub} for
equivalent definitions.

The study of Ulrich bundles has been an active area of 
research, especially since the paper of Eisenbud and Schreyer
\cite{ES}. One of the primary questions in this area is that of
the existence of Ulrich bundles on smooth projective varieties.
In particular, one would also like to know the minimal rank 
of such a bundle whenever it exists. The existence of Ulrich bundles
has been established in several cases e.g.~when $X$ is a curve \cite{ES},
a smooth complete 
intersection variety \cite{HUB}, an abelian surface \cite{AB1},
a general K3 surface \cite{AFO}, a grassmannian variety \cite{CM},
a Fano threefold of index two \cite{AB}, a ruled surface \cite{ACM}
etc. 

In this article, we investigate the existence of Ulrich bundles on
\emph{double planes}. A double plane, in our context is a smooth
projective surface $X$ which is a double cover of the projective plane
$\mbb{P}^2$. Such a surface has to necessarily be branched over a smooth
even degree curve in $\mbb{P}^2$. Refer $\mathcal{x}\,$\ref{hsurfaces}
for some well-known details about double planes.

First, we identify the necessary and sufficient conditions for a
double plane $\pi:X\ra\mbb{P}^2$ to carry an Ulrich \emph{line
  bundle} with respect to $\pi$ (i.e. a line bundle $L$ on $X$ such
that $\pi_*L=\mcOP\oplus\mcOP$) in the following
\newtheorem{thm}{Theorem}[section]
\begin{thm}\label{thm:neccsuff}
Suppose that $\pi:X\ra \mbb{P}^2$ is a double plane 
branched over a smooth curve $B$ of degree $2s$. Let $R$ denote the smooth
ramified curve in $X$ and $\sigma$ 
denote the involution of $X$ obtained by interchanging the two sheets
of the double cover. Then, the following are equivalent.
\begin{enumerate}
 \item The surface $X$ admits an Ulrich line bundle with respect to $\pi$.
 \item There is a smooth curve $D\subset X$ 
 such that $D\neq \sigma(D)$ and $D\cdot \sigma(D)=D\cdot R=\sigma(D)\cdot R =s^2$.
 \item There is a smooth curve $C$ of degree $s$ 
in $\mbb{P}^2$ which is a tangent to $B$ of even 
order at every point of $C\cap B$.
\end{enumerate}
 \end{thm}
 We prove that conditions (1) and (2) are equivalent and they
 imply condition (3) in $\mathcal{x}\,$\ref{conditions}. We then show that
 (3) implies (2) in $\mathcal{x}\,$\ref{tangent}. This proof
involves topology and certain deformation theory
arguments.

In $\mathcal{x}\,$\ref{construct}, we describe classes of smooth
double planes branched over curves of degree $2s$ for every $s\geq 1$
which admit Ulrich line bundles, and give examples. When $X$ is
branched over a conic $B$, then Ulrich line bundles arise from lines
in $\mbb{P}^2$ which are simple tangents to $B$ at any point of $B$,
as shown in
\newtheorem{corr}[thm]{Corollary}
\begin{corr}\label{conicthm}
  Let $X$ be a smooth surface which is a double cover of $\mbb{P}^2$
  branched along a smooth conic curve $B$. Then $X$ admits a pair of
  Ulrich line bundles.
\end{corr}
In fact, any double plane $X$ branched over a smooth conic is
isomorphic to $\mbb{P}^1\times\mbb{P}^1$, and the two Ulrich line
bundles correspond to the two rulings of $\mbb{P}^1\times\mbb{P}^1$,
cf. $\mathcal{x}\,$\ref{conic}.  Next, we consider double planes $X$
branched over a smooth quartic in $\mathcal{x}\,$\ref{branch quartic}
where we prove the following
\begin{corr}\label{quartic}
  Let $X$ be a smooth surface which is a double cover of $\mbb{P}^2$
  branched along a smooth quartic curve $B$. Then $X$ admits 63 pairs
  of Ulrich line bundles.
\end{corr}
Hadan \cite{IH} proves that - to any smooth plane quartic curve, there
are 63 disjoint one parameter families of smooth plane conics (simple)
tangentially meeting the quartic at 4 points. Using this, we obtain
that a general conic in each such family gives rise to a pair of
Ulrich line bundles.  When $X$ is branched over a smooth degree $2s$
curve for $s\geq 3$, we prove the theorem stated below.
\begin{thm}\label{sexticthm}
 Let $X$ be a smooth surface which is a double cover of $\mbb{P}^2$ branched along
 a generic smooth curve $B\subset\mbb{P}^2$ of degree $2s$ for $s\geq 3$. 
 Then $X$ does not admit Ulrich line bundles.
\end{thm}
We show that a necessary condition for any double plane $X$ to carry
an Ulrich line bundle is that the Picard number $\rho(X)>1$, cf. Lemma
\ref{L not ample}.  By a version of the Noether-Lefschetz theorem, it
is known that the Picard group of a double plane branched over a
generic smooth degree $2s$ curve for $s\geq 3$ is $\mathbb{Z}$
(\cite{Bu}, \cite[Chapter 2]{RF}), which proves the theorem,
cf.~$\mathcal{x}\,$\ref{s>3discussion}.

However, for each $s\geq 3$, there are special classes of double
planes which do admit Ulrich line bundles. In particular we prove the
following in $\mathcal{x}\,$\ref{s>3discussion}.
\begin{thm}\label{s>2}
Let $C\subset \mbb{P}^2$ be a smooth degree $s$ curve. Consider an 
effective divisor $\mfk{d}$ of degree $s^2$ on
$C$ corresponding to a section of $H^0(C,\mcO_C(s)\otimes L_0)$, where 
$L_0\in\emph{Pic}\,C$ is of order two.
\begin{enumerate}
 \item[(a)] Then, there exist smooth curves $B$ of degree $2s$ such that $B\cdot C=2\mfk{d}$.
 \item[(b)] Further, let $X$ be a double cover of $\mbb{P}^2$ branched
   along a smooth curve $B$ as in part (a) of this theorem. Then, $X$
   admits Ulrich line bundles.
\end{enumerate}
\end{thm}
Next, we concretely exhibit double planes branched over degree $2s$
curves with $s$ even, which admit Ulrich line bundles.
\begin{thm}\label{fermat}
 Let $B\subset \mbb{P}^2$ be a Fermat curve of degree $2s$ where
 $s$ is an even integer. Then the double plane $X$ branched along 
 $B$ admits Ulrich line bundles.
\end{thm}
We show this by identifying explicitly a degree $s$ curve $C$ tangential
to $B$ of even order at all points of $C \cap B$. We finally end with some
remarks about the existence of Ulrich line bundles in higher dimensions.
In this direction, we prove the following theorem in
$\mathcal{x}\,$\ref{ngeq3}.
\begin{thm}
 Consider a smooth $d$-fold cyclic cover $\pi:X\ra\mbb{P}^n$ 
 branched over a smooth hypersurface $B\subset\mbb{P}^n$ where $n\geq 3$. If $d\leq n$, then
 $X$ does not admit Ulrich line bundles with respect to $\pi$.
\end{thm}
The above theorem shows that, in general, we cannot expect Ulrich line 
bundles on $d$-fold cyclic coverings of $\mbb{P}^n$ when
$d\leq n$. Double planes, i.e.~the case $d = n = 2$, are the only exceptions
and are indeed special in this respect.

\section{Preliminaries}\label{prelim}
We now briefly recall some preliminaries about Ulrich bundles. 
Refer \cite{AB} for a comprehensive introduction to Ulrich
bundles.
\subsection{Ulrich bundles}\label{prelim-ub}
The definition of Ulrich bundles which we will use
in this article is as follows.
\newtheorem{defn}[thm]{Definition}
\begin{defn}\label{piulrich}
Consider a smooth projective variety $X$ admitting a finite 
 morphism $\pi:X\ra\mbb{P}^{\emph{dim}\,X}$. A vector bundle $E$
 on $X$ is said to be Ulrich with respect to the morphism $\pi$
 if the direct image $\pi_*E$ is the trivial vector bundle.
\end{defn}
The above definition is a special version of the general definition of Ulrich
bundles, which we recall now.
\begin{defn}\label{ulrich}
Let $X$ be a smooth projective variety over $\mbb{C}$, and
E be a vector bundle on $X$. Consider a very ample 
line bundle $\mcO_X(1)$ on $X$ inducing the embedding $X\hra\mbb{P}^N$.
Then $E$ is Ulrich on $(X,\mcO_X(1))$ if it satisfies one
of the following equivalent conditions.
\begin{enumerate}
 \item The direct image vector bundle $\pi_*E$ is trivial for 
 any finite linear projection $\pi$:
\begin{displaymath}
 \xymatrix{ X\ar@{^{(}->}[r]\ar[dr]_{\pi} & \mbb{P}^N\ar[d]\\
			    & \mbb{P}^{\text{dim}\,X}}
\end{displaymath}
 \item $H^i(X,E(-p))=0$ for $1\leq p\leq\emph{dim}\,X$ and for all $i$.
\end{enumerate}
\end{defn}
Note that, one can also define Ulrich bundles with respect to ample
and globally generated polarizations as in the paper \cite{AK}. In
\cite[Proposition 3]{AK}, the authors record that the existence of
Ulrich bundles with respect to an ample and globally generated line
bundle $\mcO_X(1)$ guarantees the existence of Ulrich bundles with
respect to $\mcO_X(d)$ for $d>0$.  Our Definition \ref{piulrich} is in
fact with respect to the ample and globally generated line bundle
$\mcO_X(1)=\pi^*\mcO_{\mbb{P}^{\text{dim}\,X}}(1)$.

It is also interesting to note that Ulrich bundles have several
favourable properties. For example, they are globally generated,
semistable and, if not stable, they are extensions of Ulrich bundles
of smaller rank.
\subsection{Double planes}\label{hsurfaces}
We now recall some basic facts about double planes. 

Consider a smooth curve $B\subset\mbb{P}^2 $ of even degree $2s$ which
is defined by a polynomial (section) $F\in H^0(\mbb{P}^2, \mcOP(2s))$.
We can construct a double cover $X$ of $\mbb{P}^2$ branched along $B$
as follows.  Let $\mbb{A}$ denote the total space of the line bundle
$A=\mcOP(s)$, $p:\mbb{A}\ra \mbb{P}^2$ the projection and
$t\in H^0(\mbb{A},p^*A)$ be the tautological section.  Define $X$ to
be the subvariety of $\mbb{A}$ defined by the section
$t^2-p^*F\in H^0(\mbb{A},p^*A^{\otimes
  2})=H^0(\mbb{A},p^*\mcOP(2s))$. The morphism
$p|_X=:\pi:X\ra \mbb{P}^2$ is of degree two and branched over the
curve $B$. We make note of the following observations.
\begin{enumerate}
 \item The double plane $X$ 
 \emph{is smooth if and only if} $B$ \emph{is smooth},
cf.\cite[Section on Double covers, Chapter II]{RF} and 
\cite[Proposition 4.1.6]{Laz}.
\item The inverse image $R\subset X$ of the branch curve
$B$ is isomorphic to $B$, and $\pi^*A=\pi^*\mcOP(s)=\mcO_X(R)$.
\item The direct image $\pi_*\mcO_X$ is a rank two vector bundle on
  $\mbb{P}^2$ since $\pi$ is a finite flat morphism.  The trace map
  $\text{Tr}:\pi_*\mcO_X\ra\mcOP$ gives a splitting of the natural
  morphism $\mcOP\ra\pi_*\mcO_X $.  In particular it can be shown that
  $\pi_*\mcO_X=\mcOP\oplus \mcOP(-s)$, cf. \cite[Remark 4.1.7]{Laz}.
\item The canonical bundles of $X$ and $\mbb{P}^2$ are related by the
  following relation ( cf.  \cite[$\mathcal{x}\,$1.41]{OD} ) :
$$K_X=\pi^*K_{\mbb{P}^2}\otimes\mcO_X(R)=\pi^*\mcOP(-3+s)\,.$$
\end{enumerate}
\subsection{Notation and Conventions}
\begin{enumerate}
\item By an Ulrich line bundle on the double plane $\pi:X\ra\mbb{P}^2$, we mean
an Ulrich line bundle with respect to $\pi$, i.e. a line bundle
 $L$ on $X$ such that $\pi_*L=\mcOP\oplus\mcOP$.
\item If $\mfk{d}$ is an effective divisor on a curve $C$, then
  $\text{Supp}\,\mfk{d}$ denotes the support of $\mfk{d}$.
 \item If $L$ is a line bundle on a surface $X$, then $|L|_{sm}$ is
 the set of smooth curves in the linear system $|L|$.
 \item \textbf{\underline{Tangent curves}} - Let $B$ and $C$ be two plane curves
 and $P$ be a smooth point of $B$ and $C$. Let $(B\cdot C)_P$ denote 
 the intersection multiplicity of $B$ and $C$ at $P$.
 \begin{itemize}
  \item[-] If $(B\cdot C)_P=1$, then $B$ and $C$ intersect transversally
  at the point $P$.
  \item[-] If $(B\cdot C)_P=2$, then $B$ is called a simple tangent to $C$ at $P$ or a 
  tangent of order 2 to $C$ at $P$.
  \item[-] If $(B\cdot C)_P=r>2$, then $B$ is said to be tangential of order $r$ to $C$ at $P$.
  \end{itemize}
 \end{enumerate}
\section{Conditions for the existence of Ulrich line bundles}\label{conditions}
\newtheorem{lemma}[thm]{Lemma}
In this section, we discuss the
conditions for the existence of Ulrich line bundles on double
planes. We start by making a remark regarding the global generation of
Ulrich bundles. Recall that $X$ is a double cover of $\mbb{P}^2$
branched along a smooth curve $B$ of degree $2s$.
\newtheorem{rmk}[thm]{Remark}
\begin{rmk}\label{glob gtd}
Suppose that $L$ is an Ulrich line bundle on $X$
with respect to $\pi$. Then 
$\pi_*L\simeq\mcOP\oplus\mcOP$ and so $h^0(X,L)=2$.
Note that in general an Ulrich bundle is always globally 
generated, cf. \cite[Theorem 1]{AB}. When $X$ is a 
double plane, there is a surjective morphism of locally free sheaves 
(\cite[Chapter 2, Lemma 29]{RF})
$$\mcO_X\oplus\mcO_X\simeq\pi^*\pi_*L\ra L\ra 0\,$$
which also shows the global generation of $L$ in our context.
\end{rmk}
The following lemma gives a necessary condition for a double plane to
admit Ulrich line bundles.
\begin{lemma}\label{L not ample}
 Let $X$ be a smooth double plane that admits an Ulrich
 line bundle. Then the Picard number $\rho(X)>1$ and
 in particular $\emph{Pic}\,X\not\simeq\mbb{Z}$.
\end{lemma}
\begin{proof}
Let $L$ be an Ulrich line bundle on $X$. 
Since $h^0(X,L)=2$, $L$ is a line bundle 
corresponding to the class of a non-zero effective divisor.
If $\text{Pic}\,X\simeq \mbb{Z}$, then $L$ has to be
ample.  By Remark \ref{glob gtd}, $L$ is globally 
generated and we have a surjection
$$\mcO_X\oplus \mcO_X\ra L\ra 0\,.$$
So $L$ is generated by two sections which correspond to
curves which do not intersect in the surface $X$. Thus $(L)^2=0$, implying
that $L$ is not ample. Hence $\rho(X)>1$.
\end{proof}
We now state Theorem \ref{thm:neccsuff} from the Introduction and give a proof.
\newtheorem*{thmneccsuff}{Theorem 1.1}
\begin{thmneccsuff}
Suppose that $\pi:X\ra \mbb{P}^2$ is a double plane 
branched over a smooth curve $B$ of degree $2s$. Let $R$ denote the smooth
ramified curve in $X$ and $\sigma$ 
denote the involution of $X$ obtained by interchanging the two sheets
of the double cover. Then, the following are equivalent.
\begin{enumerate}
 \item The surface $X$ admits an Ulrich line bundle with respect to $\pi$.
 \item There is a smooth curve $D\subset X$ 
 such that $D\neq \sigma(D)$ and $D\cdot \sigma(D)=D\cdot R=\sigma(D)\cdot R =s^2$.
 \item There is a smooth curve $C$ of degree $s$ 
in $\mbb{P}^2$ which is a tangent to $B$ of even 
order at every point of $C\cap B$.
\end{enumerate}
 \end{thmneccsuff}
\begin{proof}
 As $X$ is ramified 
over a smooth curve $B$ of degree $2s$,
we have $\pi_*\mcO_X=\mcOP\oplus\mcOP(-s)$. 

\paragraph{\underline{\textbf{Part 1 - $(2)\implies (3)$}}}
Consider a smooth curve $D\subset X$ as in the part (2) of the
statement of the theorem.
Let $D\cap\sigma(D)=\mfk{d}$ where
$\mfk{d}$ is an effective divisor on $D$ of degree $s^2$. Denote
$C:=\pi(D)$.  Since $D\neq \sigma(D)$, we have
$\pi^{-1}(C)=D+\sigma(D)$. The restricted morphism
$\pi|_{D+\sigma(D)}:D+\sigma(D)\ra C$ is of degree two and in
particular, it is a two sheeted covering map away from the support of
$\mfk{d}$. 
In fact, the restricted map $D\ra C$ is of degree one and
hence is a normalization map. Thereby, we have the following 
relation between the arithmetic genera of $C$ and $D$:
\begin{equation}\label{a-genus}
p_a(C)=p_a(D)+\Sigma_{p\in C}\, \delta_p\,, 
\end{equation}
where $p_a$ denotes the arithmetic genus of the curve, and
$\delta_p=\text{length}(\widetilde{\mcO_p}/\mcO_p)$, 
cf. \cite[Exercise IV.1.8]{RH1}.
 
 Since the morphism $\pi$ is of degree two,
 $\pi^*C\cdot \pi^*B=2(C\cdot B)\,.$
 But $${\pi^*C\cdot \pi^*B=(D+\sigma(D))\cdot 2R}\,.$$
 The hypothesis that $D\cdot R=\sigma(D)\cdot R=s^2$ gives
 \begin{equation}\label{intno}
 (C\cdot B)=(D+\sigma(D))\cdot R=2s^2\,. 
 \end{equation}
 Since $B$ is a degree
 $2s$ curve, the Bezout's theorem \cite[Corollary I.7.8]{RH1} shows that
 $C\in |\mcOP(s)|$, and so 
 $$p_a(C)=\frac{(s-1)(s-2)}{2}\,.$$  
 Next, from the adjunction formula for $D\subset X$, we get
$$ \omega_D=(K_{X}\otimes\mcO_X(D))|_D,\ \text{i.e.}$$
\begin{equation}\label{a-genus of D}
 2p_a(D)-2=K_X\cdot D+D^2\,.
\end{equation}
From $\mathcal{x}\,$ \ref{hsurfaces} (4), the canonical line bundle of
$X$ is $K_X=\pi^*\mcOP(s-3)$. Then,
$$K_X\cdot (D+\sigma(D))= \pi^*\mcOP(s-3)\cdot \pi^*C=\pi^*\mcOP(s-3)\cdot\pi^*\mcOP(s)=2s(s-3)\,.$$
The curves $D$ and $\sigma(D)$ are isomorphic by a global automorphism and hence
$$K_X\cdot D= s(s-3)\,.$$
In order to compute $(D)^2$, we consider $(\pi^*C)^2=2(C)^2$, giving
 $$(D+\sigma(D))^2= 2s^2\text{, i.e.}$$
 $$D^2+\sigma(D)^2+2D\cdot \sigma(D)= 2s^2\,.$$
By hypothesis, $D\cdot \sigma(D)=s^2$ and $D^2=\sigma(D)^2$ 
by the symmetry of $D$ and $\sigma(D)$.
Thus $D^2=\sigma(D)^2=0$. We thus compute the arithmetic genus of $D$ 
from \eqref{a-genus of D} to be
$$p_a(D)=\frac{(s-1)(s-2)}{2}\,.$$
Thus $p_a(C)=p_a(D)$, and from \eqref{a-genus}, we get $\Sigma_{p\in C}\, \delta_p=0$. 
Hence $C$ is a smooth plane curve, and in fact $C\simeq D\simeq \sigma(D)$.
Finally, from equation \eqref{intno}, we get the following intersection
 multiplicity for any 
 $x\in\text{supp}\,\mfk{d}$,
 $$(C\cdot B)_{\pi(x)}=(D\cdot R)_x+(\sigma(D)\cdot R)_x=2(D\cdot R)_x\,.$$
 This shows that $C$ is a tangent of an even order to the branch curve $B$ at
 every point where it intersects $B$. \emph{Thus (2) implies (3)}. 
 
\paragraph{\underline{\textbf{Part 2 - $(2)\implies (1)$}}}
From Part 1, we have $C:=\pi(D)$ is a smooth curve in $\mbb{P}^2$ with
$\trm{degree\,}C= s$ and $D^2=0$.
Consider the commutative diagram:
\begin{displaymath}
  \xymatrix{D\ar@{^{(}->}[r]^e\ar[dr]_{\sim} & D+\sigma(D) \ar@{^{(}->}[r]_{j}\ar[d]^{\pi} & X\ar[d]^{\pi}\\
    & C \ar@{^{(}->}[r]_{i} & \mbb{P}^2}
\end{displaymath}
The canonical line bundle of the smooth curve $C$ is given by the
adjunction formula,
$$\omega_C=(K_{\mbb{P}^2}\otimes\mcOP(C))|_C=i^*\mcOP(s-3)\,.$$
Since $\pi\circ e:D\ra C$ is an isomorphism, 
\begin{equation}\label{normalbundle1}
\omega_D\simeq e^*\pi^* i^*\mcOP(s-3). 
\end{equation}
Again, from
the adjunction formula $\omega_D=(K_{X}\otimes\mcO_X(D))|_D$ for $D\subset X$,
 we get,
\begin{equation}\label{normalbundle2}
\omega_D=e^*j^*\pi^*\mcOP(s-3)\otimes \mcO_X(D)|_D=e^*\pi^*i^*\mcOP(s-3)\otimes \mcO_X(D)|_D\,.
\end{equation}
Thus, by comparing equations \eqref{normalbundle1} and \eqref{normalbundle2},
we get
$$\mcO_X(D)|_D=\mcO_D\,.$$
 Consider the following short exact sequence in $X$:
 \begin{equation}\label{ses1}
0\ra \mcO_X\ra \mcO_X(D)\ra j_*e_*\mcO_D\ra 0\,,
 \end{equation}
and the associated long exact sequence of cohomology:
$$0\ra H^0(X,\mcO_X)\ra H^0(X,\mcO_X(D))\ra H^0(D,\mcO_D)\ra H^1(X,\mcO_X)\ra\cdots\,.$$
Since $\pi$ is a finite morphism,
$H^1(X,\mcO_X)\simeq H^1(\mbb{P}^2,\mcOP\oplus\mcOP(-s))=0\,.$ Hence,
$$h^0(X,\mcO_X(D))=2= h^0(\mbb{P}^2,\pi_*\mcO_X(D))\,.$$
We next pushforward \eqref{ses1} to $\mbb{P}^2$ by $\pi$ to get:
$$0\ra \pi_*\mcO_X\ra \pi_*\mcO_X(D)\ra \pi_*j_*e_*\mcO_D\ra 0\,.$$
Since 
$\pi\circ e:D\ra C$ is an isomorphism, $\pi_*j_*e_*\mcO_D\simeq i_*\mcO_C$.
So $c_1(\pi_*j_*e_*\mcO_D)=[C]=[s]$.
Also, as $\pi_*\mcO_X=\mcOP\oplus \mcOP(-s)$, we have 
$c_1(\pi_*\mcO_X)=[-s]$.
Hence, $c_1(\pi_*\mcO_X(D))=\mcO_X$. The two dimensional
space of global sections of $\pi_*\mcO_X(D)$ gives the
following exact sequence with cokernel $T$:
$$0\ra\mcOP\oplus\mcOP\ra \pi_*\mcO_X(D)\ra T\ra 0\,.$$
Now $T$ is a torsion sheaf with $c_1(T)=0$. Hence,
$\text{supp}\,T$ is a closed subset of $\mbb{P}^2$ with
codimension at least two in $\mbb{P}^2$. That is, we
have two locally free sheaves $\mcOP\oplus\mcOP$ and $ \pi_*\mcO_X(D)$ 
which are isomorphic on an open subset of $\mbb{P}^2$ whose
complement has codimension $\geq 2$. Thus, by \cite[Proposition 1.6
  (iii)]{RH}
$\pi_*\mcO_X(D)\simeq \mcOP\oplus\mcOP$.
This shows that $\mcO_X(D)$ is an Ulrich bundle for $\pi$. \emph{Hence (2)
implies (1)}.

\paragraph{\underline{\textbf{Part 3 - $(1)\implies (2)$}}}
Assume (1), i.e. let $L$ be an Ulrich line bundle on $X$. Then,
by Lemma \ref{L not ample}, $(L)^2=0$ and by Remark \ref{glob gtd}, 
$L$ is globally generated. By Bertini's theorem \cite[Corollary II.10.9]{RH1},
we can assume that $L=\mcO_X(D)$ for a smooth curve
$D\xhra{j}X$. Consider the short exact sequence
$$0\ra\mcO_X\ra\mcO_X(D)\ra j_*\mcO_X(D)|_{D}\ra 0\, ,$$
and push-forward to $\mbb{P}^2$ by $\pi$ to get
$$0\ra \mcOP\oplus\mcOP(-s)\ra \mcOP\oplus\mcOP\ra \pi_*j_*\mcO_X(D)|_{D}\ra 0\,.$$
Considering the long exact sequence of cohomology associated to the 
above short exact sequence, we get the following.
\begin{enumerate}
 \item[(a)] $h^0(D,\mcO_X(D)|_{D})=1$. Combining with the fact that 
 $$\trm{deg}\, \mcO_X(D)|_{D}=(D)^2=(L)^2=0\,$$ we get $\mcO_X(D)|_{D}=\mcO_{D}$.
 \item[(b)] Next $H^1(D,\mcO_X(D)|_{D})\simeq H^2(\mbb{P}^2, \mcOP(-s))$.
 By part (a) and Serre duality, we get 
 $H^1(D,\mcO_{D})\simeq H^0(\mbb{P}^2,\mcOP(s-3))^{\vee}$. Thereby,
 the genus of $D$ is
 $g(D)=\,h^0(\mbb{P}^2,\mcOP(s-3))=\frac{(s-1)(s-2)}{2}$.
\end{enumerate}
Now if $\sigma(D)=D$, then consider $C=D/\langle\sigma\rangle$, the quotient of
$D$ by the action of $\sigma$. We then have $\pi^*\mcOP(C)=\mcO_X(D)$ and 
by projection formula,
$$\pi_*\mcO_X(D)=\pi_*\pi^*\mcOP(C)=(\mcOP\oplus\mcOP(-s))\otimes\mcOP(C)$$
which is not trivial. This contradicts the fact that $\mcO_X(D)$ is
an Ulrich bundle.

Thus $\sigma(D)\neq D$. Setting $C:=\pi(D)$, we get $\pi^{-1}(C)=D+\sigma(D)$. 
Let $m\in \mathbb{Z}^+$ be such that $C\in |\mcOP(m)|$. 
Then $D+\sigma(D)\in |\pi^*\mcOP(m)|$, and hence 
$$K_X\cdot (D+\sigma(D))=\pi^*\mcOP(s-3)\cdot \pi^*\mcOP(m)=2m(s-3)\ \text{giving}$$
$$K_X\cdot D= m(s-3)\,.$$
The adjunction formula for $D\subset X$, i.e.
$2g(D)-2=K_X\cdot D+ D^2$ shows that
$$s(s-3)=m(s-3)\ \text{implying that } m=s\,.$$
Hence $C$ is a plane curve of degree $s$ and
$(D+\sigma(D))^2=2(C)^2=2s^2$. That is,
$$(D)^2+(\sigma(D))^2+2D\cdot\sigma(D)=2s^2\,.$$
Note that $(D)^2=(\sigma(D))^2=0$ since $D$ and $\sigma(D)$ are isomorphic curves 
in $X$. Thus $D\cdot\sigma(D)=s^2$. Further, 
$$2(2s^2)=2(C\cdot B)=\pi^*C\cdot \pi^*B=(D+\sigma(D))\cdot 2R\,.$$
Again by the symmetry of $D$ and $\sigma(D)$, we obtain
$D\cdot R=\sigma(D)\cdot R=s^2$. \emph{This
 shows that (1) $\implies$ (2)}. We show that (3) $\implies$ (2) in the 
 next section.
\end{proof}
As a consequence of the theorem, the following remark shows that 
Ulrich line bundles appear in 
pairs.
\begin{rmk}
Let $\mcO_X(D)$ be an Ulrich line bundle on $X$, where $D\subset X$
is a smooth curve. Then $\sigma(D)\subset X$ is another smooth curve
satisfying condition two of Theorem \ref{thm:neccsuff}.
Note that $\mcO_X(\sigma(D))\neq\mcO_X(D)$. For, 
if $\mcO_X(\sigma(D))=\mcO_X(D)$, then $\mcO_X(\sigma(D))\cdot\mcO_X(D)=0$.
But, we know that
$\mcO_X(\sigma(D))\cdot\mcO_X(D)=\sigma(D)\cdot D=s^2$. Hence,
$\mcO_X(\sigma(D))$ is another Ulrich line bundle on $X$.
\end{rmk}
The following remark discusses the morphism from $X$ to the 
projective line induced by an Ulrich line bundle and its fibres.
\begin{rmk}\label{oneparameter}
 Consider an Ulrich line bundle $L$ on a double plane $X$ branched
 over a smooth degree $2s$ curve $B$. Since $L$ is
 globally generated and $h^0(X,L)=2$, $L$ gives a morphism
 $$\phi:X\ra \mbb{P}^1\,.$$
 Note that the fibres of $\phi$ are curves $D\subset X$ such that
 $\mcO_X(D)=L$. 
 Hence, there is a one parameter 
 family of curves in $X$ represented by the linear equivalence class
 $[D]$, whose general element satisfies condition (2) of
 Theorem \ref{thm:neccsuff}. Thereby, we also have 
 a one parameter family of curves in $\mbb{P}^2$ satisfying condition (3)
 of the theorem, which is the image 
 of the one parameter family of curves $[D]$ in $X$.
\end{rmk}

\section{Even order Tangent curves to the branch curve}\label{tangent}
Our objective in this section is to prove that the condition (3) of 
Theorem \ref{thm:neccsuff} implies condition (2).
As earlier, suppose that $\pi:X\ra\mbb{P}^2$ is a degree two morphism of
smooth surfaces. Let $B\in |\mcOP(2s)|_{sm}$ be the smooth even degree branch 
curve of $\pi$.  Let $C\subset\mbb{P}^2$ be a curve satisfying the hypotheses
of part (3) of the theorem. That is, $C$ is a smooth degree $s$ curve
in $\mbb{P}^2$ which is a tangent to $B$ of even 
order at every point of $C\cap B$. Let $B\cap C=2\mfk{d}$, where
$\mfk{d}$ is an effective divisor of degree $s^2$ on $B$ (therefore on $C$
as well) such that $\text{Supp}\,\mfk{d}=\{P_1,P_2,\cdots, P_r\}$.
Let us also use the same notation to denote the preimages
of these points in $X$. We now prove the local reducibility of 
$\pi^{-1}(C)$.
\begin{lemma}\label{local reducible}
 The inverse image curve $\pi^{-1}(C)$ is locally reducible at each of the points
 $P_1$,$P_2,\cdots,P_{r}$ in $X$.
\end{lemma}
\begin{proof}
 Consider a local chart in $\mbb{P}^2$ around $P_i$ which maps $P_i$ to the origin 
 $(0,0)\in\mbb{C}^2$. Since both $C$ and $B$ in this affine patch
 pass through $(0,0)$, their local equations in the complete local ring
 $\mbb{C}[[x,y]]$
 contain no constant terms. By a local
 change of coordinates, set the local equation of $C$ as $x'=0$. 
 Let $l\in\mbb{N}$ be such that $(B\cdot C)_{P_i}=2l$, i.e. 
 $B$ and $C$ are tangent of order $2l$ to each other at $P_i$.
 Then in
 the coordinates $x', y$, the local equation of $B$ in $\mbb{C}[[x',y]]$ 
 can be written as
 $$x'(a+f)+y^{2l}(b+g)=0\,,$$
 where $a,b\in\mbb{C}^*$ and $f,g\in(x',y)$.
 Thus the local equation of $B$ is
 $$u_1x'+u_2y^{2l}=0\,,$$
 where we note that $u_1=(a+f),u_2=(b+g)$  are units in $\mbb{C}[[x',y]]$. 
 Since $\frac{u_2}{u_1}$ and $(\frac{u_2}{u_1})^{\frac{1}{2l}}$ are 
 units in $\mbb{C}[[x',y]]$, we change the local
 coordinates to make the local equation of $B$
 $$x'+(y')^{2l}=0\,.$$
 Finally set $x_1=x'+(y')^{2l}$. Then, in the local
 coordinates $x_1,y'$, the local equation of $B$ is
 $(x_1=0)$ and the local equation of $C$ is $(x_1-(y')^{2l}=0)$.
 
 Let $u,v$ be the local coordinates around the point
 $P_i$ in $X$. Note  that, locally around $P_i$, 
 the double cover morphism $\pi:X\ra \mbb{P}^2$ is given by 
 the following map of complete local rings :
 $$\mbb{C}[[x_1,y']]\ra\mbb{C}[[u,v]] \quad x_1\mapsto u^2, y'\mapsto v\,.$$
 Then the local equation $(x_1-(y')^{2l}=0)$
 of the curve $C$ in $\mbb{P}^2$ pulls back to
 the local curve $(u^2-v^{2l}=0)=((u-v^l)(u+v^l)=0)$ in $X$, which is reducible. 
 This shows that $\pi^{-1}(C)$ is locally reducible at $P_i$. 
\end{proof}

\begin{rmk}\label{locintmult}
  Suppose that $D_1$ and $D_2$ are curves in $X$ whose local equations
  in the completion $\mbb{C}[[u,v]]$ of the stalk $\mcO_{X,P_i}$ at
  $P_i$ are $(u-v^l=0)$ and $(u+v^l=0)$. Then the local intersection
  multiplicity of $D_1$ and $D_2$ at $P_i$ is
$$(D_1\cdot D_2)_{P_i}=\emph{dim}_{\mbb{C}}\frac{\mbb{C}[[u,v]]}{(u-v^l,u+v^l)}=\emph{dim}_{\mbb{C}}\frac{\mbb{C}[[v]]}{(v^l)}=l\,.$$
Note that the local equation of the ramified curve $R\subset X$ 
in the completion $\mbb{C}[[u,v]]$ of the stalk $\mcO_{X,P_i}$ at
  $P_i$ is $(u=0)$. Hence, for $j=1,2$ the local intersection multiplicity
$$(D_j\cdot R)_{P_i}=\emph{dim}_{\mbb{C}}\frac{\mbb{C}[[u,v]]}{(u\pm v^l,u)}=\emph{dim}_{\mbb{C}}\frac{\mbb{C}[[v]]}{(v^l)}=l\,.$$
  
\end{rmk}
Lemma \ref{local reducible} shows that, at each of the points $P_i$,
the inverse image curve $\pi^{-1}(C)$ in fact locally splits up into two components. 
We next wish to show that $\pi^{-1}(C)$ is globally reducible in $X$. 
In order to do that, we first prove the following general lemma.
\begin{lemma}\label{genus loops map trivially}
  Let $C_1$ and $C_2$ be two smooth curves in $\mbb{P}^2$. Then the
  image of the morphism of fundamental groups:
$$\pi_1(C_1\setminus C_1\cap C_2)\ra \pi_1(\mbb{P}^2\setminus C_2)$$
is generated only by the images of the loops in
$C_1\setminus C_1\cap C_2$ around the punctures.
\end{lemma}
\begin{proof}
Let $C_1\in |\mcOP(d)|_{sm}$ and $C_2\in|\mcOP(l)|_{sm}$.
The set
$$\mcD=\{[C]\in |\mcOP(d)|\,:\, C\text{ does not intersect }C_2\text{ transversally}\}\,$$
is a divisor in the projective space $|\mcOP(d)|$. Then the complement $\mcD^c$ of 
$\mcD$ is an open subset of $|\mcOP(d)|$ consisting of curves which
intersect $C_2$ transversally.

Let
$\mcC$ denote the incidence variety 
$$\mcC:=\{(x,C): C\in|\mcOP(d)|_{sm}\text{ and } x\in C \}\subset\mbb{P}^2\times |\mcOP(d)|_{sm}\,, $$
and $p:\mcC\ra |\mcOP(d)|_{sm}$ the projection. 
The fibre of $p$
over a point $[C]\in |\mcOP(d)|_{sm}$ is $p^{-1}([C])\simeq C$. By the Ehresmann's 
fibration theorem \cite[Theorem 13.1.3]{DA},
$p:\mcC\ra |\mcOP(d)|_{sm}$ is a locally trivial topological fibration.
Consequently, by denoting the open set
$p^{-1}(\mcD^{c}\cap |\mcOP(d)|_{sm})\subset\mcC$ by $\mcC^*$, we have a relative 
locally trivial fibration:
\begin{displaymath}
 \xymatrix{\mcC_{C_2}\ar@{^{(}->}[r]\ar[dr] & \mcC^*\ar[d]^p\\
	      & \mcD^c\cap |\mcOP(d)|_{sm} }
\end{displaymath}
where $\mcC_{C_2}:=\{(x,C\cap C_2)\,|\,C\in \mcD^c\cap |\mcOP(d)|_{sm}, x\in C\cap C_2\}$ 
is the relative family of divisors. 

Consider a connected and smooth 
(it is enough to assume smoothness in $\mathcal{C}^{\infty}$ sense) curve $T \subset |\mcOP(d)|$ 
and a fixed closed point $0\in T$ with the following property.
There is a family of smooth curves $\{V_t\}$  
of degree $d$ for $t\in T\setminus\{0\}$ 
which degenerates to a curve $V_0\in\mcD^c$ where $V_0$ is a degree $d$ curve in $\mbb{P}^2$ which is a union of
$d$ distinct lines $L_1,L_2,\cdots ,L_d$ which 
are concurrent at a point $O\notin C_2$.
The condition that $V_0\in\mcD^c$ implies that each $L_i$ 
intersects $C_2$ transversally at all points of $L_i\cap C_2$. 
Since $V_0\in \mcD^c$, there is a non-empty
open subset of $T\setminus\{0\}$ such that the corresponding curves
$V_t\in\mcD^c$.

By an infinitesimal deformation, the curve $C_1$ 
can be deformed to a curve $C_1'\in\mcD^c\cap |\mcOP(d)|_{sm}$.
In order to prove the claim of the lemma, it is enough to 
prove that the image of the morphism,
$$\pi_1(C_1'\setminus C_1'\cap C_2)\ra \pi_1(\mbb{P}^2\setminus C_2)$$
is generated only by the images
of the loops in $C_1'\setminus C_1'\cap C_2$ around the punctures. 

 Choose a $t_0\in T$ close to $0$ such that $V_{t_0}\in \mcD^c\cap |\mcOP(d)|_{sm}$.
 Since $C_1', V_{t_0}\in \mcD^c\cap|\mcOP(d)|_{sm}$, there is a path in $\mcD^c\cap|\mcOP(d)|_{sm}$
joining $C_1'$ and $V_{t_0}$. Note that, as $\mcC_{C_2}\hra \mcC^*$ is a relative locally trivial 
fibration over $\mcD^c\cap|\mcOP(d)|_{sm}$, $C_1'\setminus (C_1'\cap C_2)$ and
$V_{t_0}\setminus(V_{t_0}\cap C_2)$ are homotopy equivalent. Thus, it is
enough to prove that the image of 
$$\pi_1(V_{t_0}\setminus (V_{t_0}\cap C_2))\rightarrow \pi_1(\mbb{P}^2\setminus C_2)$$
is generated only by the images
of the loops in $V_{t_0}\setminus (V_{t_0}\cap C_2)$ around the punctures. 

We do this by considering the Milnor fibration \cite{JM} around $0\in T$. In
particular, choose an $\epsilon>0$ such that the open ball $B_{\epsilon}$ and
it's closure $\overline{B_{\epsilon}}$ in $\mbb{P}^2$ satisfy
$B_{\epsilon}\subset \overline{B_{\epsilon}}\subset \mbb{P}^2\setminus C_2$ and $V_{t_0}\cap B_{\epsilon}\neq\emptyset$. 
If $B_{\epsilon}^c$ denotes the complement
of the ball $B_{\epsilon}$ in $\mbb{P}^2$, then the set 
$(V_{t_0}\setminus V_{t_0}\cap C_2)\cap B_{\epsilon}^c$ is 
diffeomorphic to $\bigcup_{i=1}^d (L_i\setminus(L_i\cap C_2))\cap B_{\epsilon}^c$.
Now consider the covering by closed sets 
$$V_{t_0}\setminus(V_{t_0}\cap  C_2)\simeq (V_{t_0}\setminus 
V_{t_0}\cap C_2)\cap \overline{B_{\epsilon}}\cup\,
\bigcup_{i=1}^d (L_i\setminus(L_i\cap C_2))\cap B_{\epsilon}^c\,.$$
Denote $A:=(V_{t_0}\setminus 
V_{t_0}\cap C_2)\cap \overline{B_{\epsilon}}$ and
$A_i:=(L_i\setminus(L_i\cap C_2))\cap B_{\epsilon}^c$ for $i=1,2,\cdots,d$.
Each $A\cap A_i$ is path connected, since it is
the boundary of a disc. Hence, we apply Van Kampen theorem \cite[$\mathcal{x}\,$1.2]{Hat} 
(using closed sets)
to
obtain the fundamental group $A\cup A_1$ at a basepoint on
$A\cap A_1$:
$$\pi_1(A\cup A_1)\simeq \pi_1(A)*_{\pi_1(A\cap A_1)}\pi_1(A_1)\,.$$
Since $A\cup A_1$ is path-connected, the fundamental group
is independent of the choice of basepoint. 
Next, $(A\cup A_1)\cap A_2=A\cap A_2$ which is again path connected.
By the Van Kampen theorem, 
the fundamental group of $A\cup A_1\cup A_2$
based at a point of $A\cap A_2$ is
\begin{multline}
\pi_1(A\cup A_1\cup A_2)\simeq \pi_1(A\cup A_1)*_{\pi_1(A\cap A_2)}\pi_1(A_2)\\
\simeq \pi_1(A)*_{\pi_1(A\cap A_1)}\pi_1(A_1)*_{\pi_1(A\cap A_2)}\pi_1(A_2)\,. \nonumber
\end{multline}
Inductively, we get that $\pi_1(V_{t_0}\setminus (V_{t_0}\cap C_2))$ is isomorphic to 
$$\pi_1(A)*_{\pi_1(A\cap A_1)}
\pi_1(A_1)*_{\pi_1(A\cap A_2)}\pi_1(A_2)\cdots *_{\pi_1(A\cap A_d)}\pi_1(A_d)\,.$$
Note that by the universal property of free products, and by the 
definition of amalgamation products, we get the following commutative diagram.
\begin{displaymath}
 \xymatrix{\pi_1(A)*\pi_1(A_1)*\cdots\pi_1(A_d)\ar[r]\ar@{->>}[d] & \pi_1(\mbb{P}^2\setminus C_2)\\
	  \pi_1(V_{t_0}\setminus (V_{t_0}\cap C_2))\ar[ur] &}
\end{displaymath}
Thus it is enough to show that the image of the morphism 
$$\pi_1(A)*\pi_1(A_1)*\cdots\pi_1(A_d)\ra \pi_1(\mbb{P}^2\setminus C_2)$$
is generated by the images of loops in $V_{t_0}\setminus (V_{t_0}\cap C_2)$ around the punctures.
Note that, $A=(V_{t_0}\setminus V_{t_0}\cap C_2)\cap 
\overline{B_{\epsilon}}\subset\overline{B_{\epsilon}}$ 
is contractible in $\mbb{P}^2\setminus C_2$. 
Further, each $A_i=(L_i\setminus(L_i\cap C_2))\cap B_{\epsilon}^c$
is homotopically equivalent to a closed disc with $l$ punctures
whose fundamental group is generated by loops $\gamma_1,\gamma_2,\cdots,\gamma_l$ around the $l$
punctures and the boundary loop $\gamma_1*\gamma_2*\cdots\gamma_l$.
Each of the loops $\gamma_i$ maps to a loop $\delta_i$ around $C_2$ in $\mbb{P}^2\setminus C_2$,
and each such $\delta_i$ is homotopic to the generator of $\pi_1(\mbb{P}^2\setminus C_2) \simeq \mbb{Z}/l\mbb{Z}$.
Thus the image
of the boundary loop $\gamma_1*\gamma_2*\cdots*\gamma_l$ in $\pi_1(\mbb{P}^2\setminus C_2)$  
is $l\times(\text{generator})$ which is the trivial element of $\pi_1(\mbb{P}^2\setminus C_2)\simeq \mbb{Z}/l\mbb{Z}$. 

Thereby we get that the image of 
$$\pi_1(V_{t_0}\setminus (V_{t_0}\cap C_2))\rightarrow \pi_1(\mbb{P}^2\setminus C_2)$$
is generated only by the images
of the loops in $V_{t_0}\setminus (V_{t_0}\cap C_2)$ around the punctures. 
\end{proof}
Using notations as earlier, we prove the following theorem.
\begin{thm}\label{global reducibility}
 Let $C\subset\mbb{P}^2$ be a smooth degree $s$ curve 
 which is an even order tangent to the branch curve $B$  
at every point of $C\cap B$ . 
Then $\pi^{-1}(C)$ is a reducible curve $D_1+D_2$. 
 Further, each component $D_i$ is isomorphic to $C$, $D_2=\sigma(D_1)$
 and $D_1\cdot D_2=s^2=D_i\cdot R$, for $i=1,2$.
\end{thm}
\begin{proof}
 Recall that $R\subset X$ denotes the ramified curve of the double cover ${\pi:X\ra\mbb{P}^2}$.
Then ${\pi:X\setminus R\ra \mbb{P}^2\setminus B}$ is a genuine two sheeted covering
map. Hence we have the following morphisms between the fundamental groups.
\begin{displaymath}
 \xymatrix{         &    			&  \pi_1(C\setminus C\cap B)\ar[d]^{\Phi} & 			& \\
	    1\ar[r] & \pi_1(X\setminus R)\ar[r] & \pi_1(\mbb{P}^2\setminus B) \ar[r] &\mbb{Z}/2\mbb{Z}\ar[r] & 0}
\end{displaymath}
Suppose the above map of fundamental groups $\Phi:\pi_1(C\setminus C\cap B)\ra 
\pi_1(\mbb{P}^2\setminus B)$ lifts to $\pi_1(X\setminus R)$, then the
inclusion $C\setminus C\cap B\hra \mbb{P}^2\setminus B$ lifts to
$X\setminus R$. This gives a section of 
$\pi|_{\pi^{-1}(C\setminus C\cap B)}: \pi^{-1}(C\setminus C\cap B) \ra C\setminus C\cap B$ 
 which shows that $\pi^{-1}(C\setminus C\cap B)$ is
disconnected and thereby $\pi^{-1}(C)$ is reducible. So it is
enough to prove that the morphism $\Phi$ lifts to $\pi_1(X\setminus R$),
which is the same as 
showing that the image in $\mbb{Z}/2\mbb{Z}$ of $\Phi(\pi_1(C\setminus C\cap B))$ 
is the trivial element.

 By Lemma \ref{genus loops map trivially}, $\Phi(\pi_1(C\setminus C\cap B))$
 is generated only by loops in $C\setminus C\cap B$ around the punctures (i.e. points of $C\cap B$).
 Let $x\in C\cap B$ and let $\gamma_x$ be a loop in $C\setminus C\cap B$ around $x$. Then under $\Phi$,
 $$ \gamma_x\mapsto (C\cdot B)_x\delta_x\,$$
 where $\delta_x$ is a small loop around $B$ at $x$. But such a $\delta_x$ can be
 considered a 
 generator of $\pi_1(\mbb{P}^2\setminus B)$ and $(C\cdot B)_x$ is even by 
 hypothesis. Hence $\gamma_x$ maps to an even multiple of the generator
 of $\pi_1(\mbb{P}^2\setminus B)$, whose further image in $\mbb{Z}/2\mbb{Z}$ 
 is trivial. This shows that $\pi^{-1}(C)$ is reducible and has the form $D_1+D_2$.
 
 Now $\pi|_{D_1+D_2}:D_1+D_2\ra C$ is a degree two finite map. Since 
 $C\cap B=2\mfk{d}$ where $\mfk{d}$ is a divisor on $B$ of
 degree $s^2$, the map $\pi|_{D_1+D_2}$ is ramified exactly along the inverse
 image of $\mfk{d}$. Away from the ramification points $\pi|_{D_1+D_2}$ is
 a covering map. Hence each $D_i$ is isomorphic to $C$.
 Thus, $D_2=\sigma(D_1)$ and $\pi^{-1}(C)=D_1+\sigma(D_1)$. From Remark 
 \ref{locintmult}, we see that $D_1\cdot D_2=\Sigma (D_1\cdot D_2)_{P_i}=s^2$ where
 the summation is taken over points $P_i$ in the inverse image of the divisor
 $\mfk{d}$. Similarly, the local intersection multiplicity computation in
 Remark \ref{locintmult} also shows that $D_i\cdot R=s^2$, for $i=1,2$.
\end{proof}
\begin{proof}[Proof of Theorem \ref{thm:neccsuff} ((3)$\implies$ (2))]
 Theorem \ref{global reducibility} proves that the condition (3) of 
 Theorem \ref{thm:neccsuff} implies condition (2).
\end{proof}

\section{Ulrich line bundles on double planes}\label{construct}
In this section, we study the existence of Ulrich line
bundles over smooth double planes. We describe
classes of double planes $X$ branched along
a smooth degree $2s$ curve $B$ and Ulrich line bundles over
them, for every $s$. For a fixed $B\in|\mcOP(2s)|_{sm}$, denote
\begin{equation}\label{tangentfiber}
 \mfk{T}_B=\{C\in |\mcOP(s)|_{sm}\,|\,B\cap C=2\mfk{d},\ \mfk{d}\text{ an effective divisor of degree } s^2\text{ on }B\} \,.
\end{equation}
Then $\mfk{T}_B$ consists of smooth degree $s$ curves $C$ which are 
even order tangents to $B$ at every point in $ C\cap B $. 
By Theorem \ref{thm:neccsuff}, the non-emptiness of $\mfk{T}_B$
implies the existence of Ulrich line bundles on $X$. Note that,
by Remark \ref{oneparameter}, if $\mfk{T}_B$ is non-empty, then
it contains at least a one parameter family of curves.
\begin{rmk}\label{sard}
 While defining the set $\mfk{T}_B$ in \eqref{tangentfiber}, we could assume more generally that 
 the divisor $\mfk{d}\subset B$ is reduced. 
 Indeed, consider an Ulrich
 line bundle $\mcO_X(D)$ on $X$ coming from a $C\in \mfk{T}_B$.
 This line bundle gives a morphism $\phi:X\ra\mbb{P}^1$, whose fibres are precisely 
 the one parameter family of divisors in $X$ which are elements of $|D|$. 
 Consider the restriction of $\phi$ to the ramified curve $R$ i.e. $\phi|_R:R\ra\mbb{P}^1$,
 which is finite of degree $s^2$. 
 By Sard's Lemma, the general fibre which is of the form $D\cap R$ of $\phi|_R$ is non-singular. Hence, 
 a general $D\cap R$ is a reduced divisor on $R$ consisting of $s^2$ points. 
 Since $C\cap B=\pi(D\cap R)$, for a general $C$, the set
 $C\cap B$ consists of precisely $s^2$ points, and the corresponding divisor 
 $\mfk{d}$ (whose support is $C\cap B$) is reduced.
 \end{rmk}

\subsection{\underline{The branch curve is a conic ($s=1$)}}\label{conic}
It is known that a double plane $X$ branched over a
smooth conic is isomorphic to $\mbb{P}^1\times\mbb{P}^1$.
So, $\text{Pic}\,X=\mbb{Z}\oplus\mbb{Z}$. Consider generators
$l$ and $m$ of type $(0,1)$ and $(1,0)$.
\newtheorem*{corrconic}{Corollary 1.2}
\begin{corrconic}
 Let $X$ be a smooth surface which is a double cover of $\mbb{P}^2$ branched along
 a smooth conic curve $B$. Then $X$ admits a pair of Ulrich line bundles.
\end{corrconic}
\begin{proof}
Suppose $\pi:X\ra \mbb{P}^2$ is a smooth double cover branched along a conic
$B\subset \mbb{P}^2$. 
In this case
\begin{equation*}
 \mfk{T}_B=\{C\in |\mcOP(1)|_{sm}\,|\,B\cap C=2\mfk{d}\,, \mfk{d}\text{ an effective divisor of degree } 1 \text{ on }B\} \,.
\end{equation*}
That is, $\mfk{T}_B$
consists of lines $C\subset\mbb{P}^2$ which are simple tangents to
the conic $B$ at any point $P\in B$. This set $\mfk{T}_B$ is
clearly non-empty. Hence, by Theorem \ref{thm:neccsuff}, the double plane $X$ admits
Ulrich line bundles. 

For $P\in B$, let $C_P$ denote a line in $\mbb{P}^2$
tangent to $B$ at $P$. Then the inverse image of $C_P\in\mfk{T}_B$, 
is $\pi^{-1}(C_P)=D_P+D_P'$, where $D_P$ and $D_P'$ are isomorphic to $C_P$ and 
$D_P\cdot D_P'=1$ by Theorem \ref{global reducibility}.
This also gives $D_P^2=D_P'^2=0$. Hence
$D_P,D_P'$ correspond to divisors in the linear systems
$|l|$ and $|m|$ (i.e. corresponding to the rulings
$(0,1)$ and $(1,0)$). Hence $X$ admits precisely
two Ulrich bundles given by $\mcO_X(l)$ and $\mcO_X(m)$.
\end{proof}
\subsection{\underline{The branch curve is  a quartic ($s=2$)}}\label{branch quartic}
A smooth projective surface  $X$ 
branched along a smooth quartic curve $B$ is a Del Pezzo surface
obtained by blowing up $\mbb{P}^2$ at 7 points in general position.
Thus $\text{Pic}\,X\simeq \mbb{Z}^{\oplus\, 8}$. 

Hadan, in \cite[$\mathcal{x}\,$2.1, $\mathcal{x}\,$2.2]{IH}, analyses the tangent conics 
to a given smooth quartic $B$. He defines a tangent conic to $B$ 
to be a smooth conic which intersects $B$ at 4 points each being a point
of simple tangency. From the 
short exact sequence,
$$0\ra \mcOP(-2)\ra\mcOP(2)\ra\mcO_B(2)\ra 0\,,$$
we obtain that $H^0(\mbb{P}^2,\mcOP(2))\simeq H^0(B,\mcO_B(2))$.
Hence any conic is determined by its restriction to $B$. Using this basic fact,
Hadan studies conics corresponding to divisors of the form $2\mfk{d}$, where
$\mcO_B(2\mfk{d})=\mcO_B(2)$ and $\mfk{d}$ is an effective divisor on $B$ 
of degree $4$. It is shown that to any smooth plane quartic,
the variety of tangent conics is a union of 63 disjoint one parameter families,
cf. \cite[Proposition 2.1]{IH}.
\newtheorem*{corrquartic}{Corollary 1.3}
\begin{corrquartic}
 Let $X$ be a smooth surface which is a double cover of $\mbb{P}^2$ branched along
 a smooth quartic curve $B$. Then $X$ admits 63 pairs of Ulrich line bundles.
\end{corrquartic}
\begin{proof}
In this case, $\mfk{T}_B$ is given by
\begin{equation}\label{tangents}
 \{C\in|\mcOP(2)|_{sm}\ |\ B\cap C=2\mfk{d}, \mfk{d}\text{ an effective divisor of degree } 4\text{ on }B\}\,.
\end{equation}
 By the discussion above, $\mfk{T}_B$ is non-empty and in fact 
 consists of 63 disjoint one parameter families. By Theorem 
 \ref{global reducibility}, for each $C\in \mfk{T}_B$, the 
 inverse image can be written as $\pi^{-1}(C)=D+\sigma(D)$.
 In fact, each one parameter family of tangent conics in $\mfk{T}_B$ gives rise to 
 two classes $[D]$ and $[\sigma(D)]$ of curves in $X$, and hence
 two Ulrich line bundles $\mcO_X(D)$ and $\mcO_X(\sigma(D))$ on $X$.
 Hence, $X$ admits 63 pairs of Ulrich line bundles. 
\end{proof}
\begin{rmk}
 When $s=1$ and $s=2$, we can prove that the inverse image
 $\pi^{-1}(C)$ for a $C\in\mfk{T}_B$ is reducible without
 appealing to Theorem \ref{global reducibility}. When 
 $s=1,2$, from the genus-degree formula, a smooth $C\simeq \mbb{P}^1$. 
For any $C$, consider the inverse image curve $\pi^{-1}(C)$ in $X$. The
normalization $\widetilde{\pi^{-1}(C)}$ of $\pi^{-1}(C)$ 
is an etale double cover of $C$. But there are no
irreducible etale double covers of $\mbb{P}^1$. The only etale
double cover of $\mbb{P}^1$ is of the form $\mbb{P}^1\sqcup\mbb{P}^1$.
Thereby, $\widetilde{\pi^{-1}(C)}$
and hence $\pi^{-1}(C)$ are both reducible. This proves that $\pi^{-1}(C)$ is of the 
form $D_1+D_2$ where $D_i$'s are smooth curves isomorphic to $C$.  
However, the same argument will not be valid when $s>2$, since 
we will be dealing with curves of higher genus and their etale covers.
\end{rmk}
\subsection{\underline{The branch curve is of degree $2s$ when $s>2$}}\label{s>3discussion}
When $s=3$, then we can check that $K_X\simeq\mcO_X$ and
$H^1(X,\mcO_X)=0$, whereby $X$ is a K3 surface. For $s>3$,
 it is known that $X$ is a surface of general type. 
 The proof of Theorem \ref{sexticthm} is an immediate consequence of Lemma
 \ref{L not ample}. 
 \newtheorem*{thmdegreelarge}{Theorem 1.4}
 \begin{thmdegreelarge}
 Let $X$ be a smooth surface which is a double cover of $\mbb{P}^2$ branched along
 a generic smooth curve $B\subset\mbb{P}^2$ of degree $2s$ for $s\geq 3$. 
 Then $X$ does not admit Ulrich line bundles.
\end{thmdegreelarge}
\begin{proof}
 When $X$ is a generic double plane branched over a smooth
 curve $B$ of degree $2s$, with $s>2$, then $\text{Pic}\,X\simeq\mbb{Z}$,
 cf. \cite[Chapter II]{RF}, \cite{Bu}. Thus by Lemma \ref{L not ample}, 
 a generic double plane branched over
 a degree $2s$ curve for $s>2$ does not admit Ulrich line bundles. 
\end{proof}
In particular, the above theorem implies that for a generic
$B\in|\mcOP(2s)|_{sm}$ with $s>2$, we have $\mfk{T}_B=\emptyset$. Thus
starting from a general smooth degree $2s$ curve ($s>2$), one cannot
find smooth degree $s$ curves everywhere tangent to it to an even order.
However, we now show that starting with a smooth degree $s$ curve $C$,
we can find smooth degree $2s$ curves $B$ which are even 
order tangential to $C$ at all points of $C\cap B$. Thus, the double planes branched over such $B$-s
admit Ulrich line bundles.

\newtheorem{prop}[thm]{Proposition}
\begin{prop}\label{tangentbranch}
 Let $C$ be a smooth curve of degree $s$ in $\mbb{P}^2$. Then 
 there 
 exist smooth degree $2s$ curves $B$ such that $B\cdot C=2\mfk{d}$
 where $\mfk{d}$ is an effective divisor of degree $s^2$ on $C$.
 That is, there exist smooth curves $B$ of degree $2s$ such that
 $B$ is tangential
 to $C$ to an even order at every point of $B\cap C$.
\end{prop}
\begin{proof}
 Suppose that $C$ is the zero set of the degree $s$ homogeneous polynomial
 $F(x,y,z)$ where $x,y,z$ are the homogeneous coordinates in $\mbb{P}^2$.
 Let us denote the inclusion $C\hra \mbb{P}^2$ by $i$.
 Now consider the following short exact sequence:
 $$0\ra \mcOP(-s)\xra{\otimes\,F}\mcOP\ra i_*\mcO_C\ra 0\,.$$
 Tensoring by $\mcOP(2s)$, we get
 $$0\ra \mcOP(s)\xra{\otimes\,F}\mcOP(2s)\ra i_*\mcO_C(2s)\ra 0\,.$$
 The long exact sequence of cohomology associated to the 
 above short exact sequence gives
 $$0\ra H^0(\mbb{P}^2,\mcOP(s))\xra{\otimes F} H^0(\mbb{P}^2,\mcOP(2s))
 \ra H^0(C,\mcO_C(2s))\ra 0\,.$$
 Therefore, 
 $$H^0(C,\mcO_C(2s))\simeq\frac{H^0(\mbb{P}^2,\mcOP(2s))}{F.H^0(\mbb{P}^2,\mcOP(s))}\,.$$
If $L_0\in\text{Pic}\,C$ is an element of order two, then 
the line bundle $\mcO_C(s)\otimes L_0$ on $C$ is of degree $s^2$, 
and the Riemann-Roch
formula gives that $$h^0(C,\mcO_C(s)\otimes L_0)=\frac{s^2+3s}{2}\,.$$
 
For $\mfk{d}\in |\mcO_C(s)\otimes L_0|$, 
consider the corresponding divisor 
 $2\mfk{d}\in |\mcO_C(2s)|$. This corresponds to an element 
 $\overline{G}\in \frac{H^0(\mbb{P}^2,\mcOP(2s))}{F.H^0(\mbb{P}^2,\mcOP(s))}$ 
 upto a non-zero scalar.
  Thus any element in the vector subspace 
 $$V=\{\lambda G+FH: \lambda\in\mbb{C},\ H\in H^0(\mbb{P}^2,\mcOP(s))\}$$ of
 $ H^0(\mbb{P}^2,\mcOP(2s))$ with $\lambda\neq 0$ maps to an element in $H^0(C,\mcO_C(2s))$ 
 corresponding to $2\mfk{d}$. 
 
 This means that any such degree $2s$ curve $B$ in the linear system
 $\mbb{P}V$ meets $C$ tangentially along the divisor $\mfk{d}$ 
 and at each point in $\text{Supp}\,\mfk{d}$, $B$ is an even
 order tangent to $C$. In order to 
 prove our claim, we need to show the existence of \emph{smooth} curves $B$.
 
 By Bertini's
 theorem, a general element of $\mbb{P}V$ is smooth away from the base points
 of this linear system. The base points of $\mbb{P}V$ are precisely
 the points in the support of $\mfk{d}$. Let $P_1,P_2,\cdots, P_r$ be
 these points. As $P_1$ is a smooth point of $C$, the partial derivatives of $F$
 cannot all vanish at $P_1$. Without loss of generality, assume that 
 $F_x(P_1)\neq 0$. Consider the set of all $H\in H^0(\mbb{P}^2,\mcOP(s))$ 
 with $(G+FH)_x\neq 0$, i.e.
 $$U_1=\{H\in H^0(\mbb{P}^2,\mcOP(s))\,:\, H(P_1)\neq 0, H(P_1)\neq \frac{-G_x(P_1)}{F_x(P_1)}\}\,.$$
 Note that $U_1$ is an open dense subset of $H^0(\mbb{P}^2,\mcOP(s))$.
 Similarly consider open subsets $U_i$  corresponding to the points $P_i$ for $i=2,3,\cdots, r$. For a general
 $H\in \cap_{i=1}^{r} U_i$, the curve corresponding to $G+FH$ is a 
 smooth curve of degree $2s$ which meets $C$ (even order) tangentially 
 at every point of $B\cap C$.
\end{proof}
As remarked earlier, the above proposition gives rise to special
double planes
branched over degree $2s$ curves ($s>2$) that carry Ulrich line bundles.
Using this we prove Theorem \ref{s>2}.
\newtheorem*{thmspecial}{Theorem 1.5}
\begin{thmspecial}
Let $C\subset \mbb{P}^2$ be a smooth degree $s$ curve. Consider an 
effective divisor $\mfk{d}$ of degree $s^2$ on
$C$ corresponding to a section of $H^0(C,\mcO_C(s)\otimes L_0)$, where 
$L_0\in\emph{Pic}\,C$ is of order two.
\begin{enumerate}
 \item[(a)] Then, there exist smooth curves $B$ of degree $2s$ such that $B\cdot C=2\mfk{d}$.
 \item[(b)] Further, let $X$ be a double cover of $\mbb{P}^2$ branched along a smooth curve
$B$ as in part (a) of this theorem. Then, $X$ admits Ulrich line bundles.
\end{enumerate}
\end{thmspecial}
\begin{proof}
Proposition \ref{tangentbranch} proves Part (a) of the Theorem. Suppose 
$C$ is a smooth degree $s$ curve and $B$ is a smooth degree $2s$ curve that is 
even order tangent to $C$ everywhere. 
Let $X$ be a double plane branched along $B$. Then,
 $C\in\mfk{T}_B$.
 By Theorem \ref{global reducibility}, $\pi^{-1}(C)=D_1+D_2$, 
 where $D_2=\sigma(D_1)$ and $D_1\cdot D_2=s^2=D_i\cdot R$. 
 Thus by Theorem \ref{thm:neccsuff}, $\mcO_X(D_1)$ and $\mcO_X(D_2)$ are
 Ulrich line bundles on $X$.
\end{proof}
We now give a proof of Theorem \ref{fermat}, which explicitly 
gives examples of double planes $\pi:X\ra\mbb{P}^2$
branched over a degree $2s$ curve, for each \emph{even} $s> 1$, which admit
Ulrich line bundles.
\newtheorem*{thmfermat}{Theorem 1.6}
\begin{thmfermat}
 Let $B\subset \mbb{P}^2$ be a Fermat curve of degree $2s$ where
 $s$ is an even integer. Then the double plane $X$ branched along 
 $B$ admits Ulrich line bundles.
\end{thmfermat}
\begin{proof}
 Let $B\subset \mbb{P}^2$ be a Fermat curve of degree $2s$. Then there exists a choice
 of coordinates $x,y,z$ on $\mbb{P}^2$ in which $B$ has the form
 $$ (x^{2s}-y^{2s}-z^{2s}=0)\,.$$ Consider the double plane 
$\pi:X\ra\mbb{P}^2$ branched over this smooth curve $B$. 

Let $C$ be the degree $s$ Fermat
 curve $(x^s-y^s-z^s=0)$. 
 Denote the $s$-th roots of unity by
 $1,\xi,\xi^2,\cdots,\xi^{s-1}$. Then, for $a\neq 0$, at the $2s$-points
 $(a:a:0), (a:\xi a:0),\cdots, (a:\xi^{s-1}a:0)$ and
 $(a:0:a), (a:0:\xi a),\cdots, (a:0:\xi^{s-1}a)$ respectively,
 the lines $(x-y=0), (\xi x- y=0), \cdots, (\xi^{s-1} x-y=0)$
 and the lines $(x-z=0), (\xi x- z=0), \cdots, (\xi^{s-1} x-z=0)$
 are :
 \begin{itemize}
  \item tangents of order $2s$ to the curve $B$, and
  \item tangents of order $s$ to the curve $C$.
 \end{itemize}
Hence, at each of above $2s$ points, $B$ and $C$ are tangential
of order $s$ to each other. This accounts for all the intersections
between $B$ and $C$ since $B\cdot C=2s^2$. 
Thereby the curve $C$ 
meets $B$ at $2s$ points, each of which is a point of tangency
of even order $s$ (since $s$ is even by hypothesis). 
Thus by Theorem \ref{global reducibility},                                                                                                                                                                                                             
$\pi^{-1}(C)=D_1+D_2$ and each $\mcO_X(D_i)$ is an Ulrich bundle
on $X$.
\end{proof}
\section{Ulrich line bundles on higher dimensional cyclic covers of Projective space}\label{ngeq3}
In this section, we make a few remarks about the existence of Ulrich 
line bundles on higher dimensional cyclic covers.
\newtheorem*{thmhighdim}{Theorem 1.7}
\begin{thmhighdim}\label{higherdim}
 Consider a smooth $d$-fold cyclic cover $\pi:X\ra\mbb{P}^n$ 
 branched over a smooth hypersurface $B\subset\mbb{P}^n$ where $n\geq 3$. If $d\leq n$, then
 $X$ does not admit Ulrich line bundles with respect to $\pi$.
\end{thmhighdim}
\begin{proof}
Let $\pi:X\ra\mbb{P}^n$  be a smooth $d$-fold cyclic cover 
branched over a smooth hypersurface where $n\geq 3$ and $d\leq n$.
If $L$ is an Ulrich line bundle on $X$
, then $h^0(X,L)=d$ and $L$ is 
globally generated. That is, there is a surjective morphism
$$\mcO_X^{\oplus\,d}\ra L\ra 0\,,$$
which shows that there are $d$ divisors of $X$ in $|L|$
which do not intersect, i.e. $L^d=0$.

\begin{itemize}
 \item \underline{Double cover of $\mbb{P}^3$} - 
 If $X$ is a double cover of $\mbb{P}^3$
admitting an Ulrich line bundle, then $L^2=0$, 
that is, there are two hypersurfaces $D_1$ and $D_2$ in 
$X$ with $D_1\cap D_2=\emptyset$. But this is not possible.
Indeed, let $\pi(D_i)$ denote the hypersurfaces in $\mbb{P}^3$ which are the 
images of $D_i$ for $i=1,2$. Note that
$$\pi(D_1)\cap\pi(D_2)\cap B\neq\emptyset\,,$$
since the images of $n$ hypersurfaces in $\mbb{P}^n$ is not empty.
This shows that $D_1\cap D_2\neq\emptyset$ and in fact they have
non-empty intersection along the ramification divisor. Thus, $X$ cannot
admit Ulrich line bundles.
\item \underline{3-fold cover of $\mbb{P}^3$ branched along a cubic} - 
When $X$ is a threefold cover of $\mbb{P}^3$ 
branched along a cubic, then $X$ is a cubic hypersurface in $\mbb{P}^4$.
By the Noether-Lefschetz theorem
\cite[Example 3.1.25]{Laz}, $\text{Pic}\,X=\mbb{Z}$, and 
hence any non-zero effective divisor is ample.
Thus $X$ cannot admit a globally generated line bundle $L$ with
$L^3=0$. So, $X$ does not admit Ulrich line bundles.
\item \underline{Other cases}- 
The remaining cases are:
\begin{itemize}
 \item[-] $X$ is a threefold covering of $\mbb{P}^3$ branched
 along a smooth hypersurface of degree $3s$ with $s\geq 2$;
 \item[-] $X\ra\mbb{P}^n$ be a $d$-fold
cover of $\mbb{P}^n$ with $n\geq 4$ and $d\leq n$. 
\end{itemize}
By discussion in the previous case, 
if $L$ is a line bundle on $X$ with $L^d=0$ and $h^0(X,L)>0$,
then $\text{Pic}\,X\neq \mbb{Z}$. 
By applications of Noether-Lefschetz theorem 
 to the pairs $(X,R)$ and $(\mbb{P}^n,B)$ and the isomorphism of $R$ and $B$,
 \cite{LP} prove that $\text{Pic}\,X=\mbb{Z}$
 in both these cases. Thus, $X$ cannot admit Ulrich line bundles.
 \end{itemize}
\end{proof}

\subsection*{Acknowledgements} We would like to thank N. Mohan Kumar
for useful comments on the initial version of the draft. This work was
supported by the Department of Atomic Energy, Government of India
[project no. 12 - R\&D - TFR - 5.01 - 0500].

\bibliographystyle{plain}
\bibliography{refs}

\end{document}